\documentclass[a4paper,12pt]{article}
\usepackage[top=2.5cm,bottom=2.5cm,left=2.5cm,right=2.5cm]{geometry}
\usepackage{cite, amsmath, amssymb}
\usepackage[margin=1cm,%
            font=small,%
            format=hang,%
            labelsep=period,%
            labelfont=bf]{caption}
\usepackage{latexsym,enumerate}
\usepackage{amsmath,amsthm,amsopn,amstext,amscd,amsfonts,amssymb}
\usepackage[ansinew]{inputenc}
\usepackage{verbatim}
\usepackage{graphicx}
\usepackage{pstricks}
\usepackage{wasysym}
\usepackage[all]{xy}
\usepackage{epsfig} 
\usepackage{graphicx,psfrag} 
\usepackage{subfigure}
\usepackage{pstricks,pst-node}

\usepackage{multicol}
\usepackage{color}
\usepackage{colortbl}

\newcommand{\RR}{\mathbb{R}}
\newcommand{\ZZ}{\mathbb{Z}}

\newtheorem{theorem}{Theorem}[section]

\newtheorem{proposition}[theorem]{Proposition}

\newtheorem{definition}[theorem]{Definition}

\newtheorem{remark}[theorem]{Remark}

\newcommand{\spb}[1]{\smallskip}
\newcommand{\mpb}[1]{\medskip}
\newcommand{\bpb}[1]{\bigskip}

\renewcommand{\a}{\alpha}

\renewcommand{\d}{\delta}
\newcommand{\D}{\Delta}

\newcommand{\G}{\Gamma}

\begin{document}

\DeclareGraphicsExtensions{.jpg,.pdf,.mps,.png}

\title{\vspace{6cm} New Lower Bounds for the First Variable Zagreb Index}

\author
{Alvaro Mart\'{\i}nez-P\'erez$^a$, Jos\'e M. Rodr{\'\i}guez$^b$}



\maketitle{}


\centerline{\emph{$^a$Facultad de Ciencias Sociales, Universidad de Castilla-La Mancha,}}
\centerline{\emph{Avda. Real Fábrica de Seda, s/n. 45600 Talavera de la Reina, Toledo, Spain}}
\centerline{e-mail: alvaro.martinezperez@uclm.es}

\

\centerline{\emph{$^b$Departamento de Matem\'aticas, Universidad Carlos III de Madrid,}}
\centerline{\emph{Avenida de la Universidad 30, 28911 Legan\'es, Madrid, Spain}}
\centerline{e-mail: jomaro@math.uc3m.es}

\

\centerline{}

\thispagestyle{empty}

\begin{abstract}
The aim of this paper is to obtain new sharp inequalities for a large family of topological indices, including the first variable Zagreb index $M_1^\a$, and to characterize the set of extremal graphs with respect to them.
Our main results provide lower bounds on this family of topological indices involving just the minimum and the maximum degree of the graph.
These inequalities are new even for the first Zagreb, the inverse and the forgotten indices.
\end{abstract}

%

\baselineskip=0.30in

\section{Introduction}

A topological descriptor is a single number that represents a chemical structure in graph-theoretical terms via the
molecular graph, they play a significant role in mathematical chemistry especially in the QSPR/QSAR investigations.
A topological descriptor is called a topological index if it correlates with a molecular property.
Topological indices are used to understand physicochemical properties of chemical compounds,
since they capture some properties of a molecule in a single number.
Hundreds of topological indices have been introduced and studied, starting with the
seminal work by Wiener \cite{Wi}.


Topological indices based on end-vertex degrees of edges have been
used over 40 years. Among them, several indices are recognized to be useful tools in
chemical researches.
Probably, the best know such descriptor is the Randi\'c connectivity
index ($R$) \cite{R}. 

Two of the main successors of the Randi\'c index are the first and second Zagreb indices,
denoted by $M_1$ and $M_2$, respectively, and introduced by Gutman and Trinajsti\'c in $1972$ (see \cite{GT}).
They are defined as
$$
M_1(G) = \sum_{u\in V(G)} d_u^2,
\qquad
M_2(G) = \sum_{uv\in E(G)} d_u d_v ,
\qquad
$$
where $uv$ denotes the edge of the graph $G$ connecting the vertices $u$ and $v$, and
$d_u$ is the degree of the vertex $u$.
Along the paper, we will denote by $m$ and $n$, the cardinality of the sets
$E(G)$ and $V(G)$, respectively.

There is a vast amount of research on the Zagreb indices.
For details of their chemical applications and mathematical theory see \cite{Gutman}, \cite{GD}, \cite{GR}, \cite{NKMT}, and the references therein.

In \cite{LZheng}, \cite{LZhao}, \cite{MN}, the \emph{first and second variable Zagreb indices} are defined as
$$
M_1^{\a}(G) = \sum_{u\in V(G)} d_u^{2\a},
\qquad
M_2^{\a}(G) = \sum_{uv\in E(G)} (d_u d_v)^\a ,
$$
with $\a \in \RR$.

Note that $M_1^{0}$ is $n$, $M_1^{1/2}$ is $2m$, $M_1^{1}$ is the first Zagreb index $M_1$, $M_1^{-1/2}$ is the inverse index $ID$ \cite{Faj},
$M_1^{3/2}$ is the forgotten index $F$, etc.;
also, $M_2^{0}$ is $m$, $M_2^{-1/2}$ is the usual Randi\'c index, $M_2^{1}$ is the second Zagreb index $M_2$,
$M_2^{-1}$ is the modified Zagreb index \cite{NKMT}, etc.

The concept of variable molecular descriptors was proposed as a new way of
characterizing heteroatoms in molecules (see \cite{R2}, \cite{R3}), but also to assess the structural differences (e.g.,
the relative role of carbon atoms of acyclic and cyclic parts in alkylcycloalkanes \cite{RPL}).
The idea behind the variable molecular descriptors is that the variables are determined during the
regression so that the standard error of estimate for a particular studied property is as small as possible (see, e.g., \cite{MN}).

In the paper of Gutman and To\v{s}ovi\'c \cite{Gutman8}, the correlation abilities of $20$ vertex-degree-based topological indices
occurring in the chemical literature were tested for the case of standard
heats of formation and normal boiling points of octane isomers.
It is remarkable to realize that the second variable Zagreb index $M_2^\alpha$
with exponent $\alpha = -1$ (and to a lesser extent with exponent $\alpha = -2$)
performs significantly better than the Randi\'c index ($R=M_2^{-0.5}$).

The first variable Zagreb index has been used to study molecular complexity, chirality, ZE-isomerism and hetero-systems.
Overall, variable Zagreb indices exhibit a potential applicability for deriving multi-linear regression models \cite{Drmota}.

Various properties, relations and uses of the first variable Zagreb index are discussed in several papers (see, e.g., \cite{LL}, \cite{ZWC}, \cite{ZZ}).

Throughout this work, $G=(V (G),E (G))$ denotes a (non-oriented) finite simple (without multiple edges and loops) non-trivial ($E(G) \neq \emptyset$) graph.
A main topic in the study of topological indices is to find bounds of the indices involving several parameters.
The aim of this paper is to obtain new sharp inequalities for a large family of topological indices and to characterize the set of extremal graphs with respect to them.
Our main results provide lower bounds on this family of topological indices involving just the minimum and the maximum degree of the graph
(see Theorems  \ref{t:tilde1} and \ref{t:tilde6}).
This family of indices includes, among others, the first variable Zagreb index $M_1^\a$ (see Section 4).
The inequalities obtained are new even for the first Zagreb, the inverse and the forgotten indices.

\section{Minimum and maximum degree}

Given integers $1 \le \delta \le \Delta$, let us define $\mathcal{G}_{\d,\D}$ as the set of graphs $G$ with minimum degree $\delta$, maximum degree $\Delta$ and such that:

$(1)$ $G$ is isomorphic to the complete graph with $\D+1$ vertices $K_{\D+1}$, if $\d = \D$,

$(2)$ $|V(G)|=\D+1$ and there are $\D$ vertices with degree $\d$, if $\d < \D$ and $\Delta (\delta+1)$ is even,

$(3)$ $|V(G)|=\D+1$ and there are $\D-1$ vertices with degree $\d$ and a vertex with degree $\d+1$, if $\d < \D-1$ and $\Delta (\delta+1)$ is odd,

$(4)$ $|V(G)|=\D+1$ and there are $\D-1$ vertices with degree $\d$ and two vertices with degree $\D$, if $\d = \D-1$ and $\Delta$ is odd (and thus $\Delta (\delta+1)$ is odd).

\begin{remark} \label{remark sharp0}
Every graph $G \in \mathcal{G}_{\d,\D}$ has maximum degree $\Delta$ and $|V(G)|=\D+1$.
Hence, every graph $G \in \mathcal{G}_{\d,\D}$ is connected. 
\end{remark}

Let us recall the following proposition from \cite{MR}:

\begin{proposition}\label{prop edges}
For any integers $1 \le \delta \le \Delta$, we have $\mathcal{G}_{\d,\D} \neq \emptyset$.
Let $G$ be a graph with minimum degree $\delta$ and maximum degree $\Delta$. Then
$$
|E(G)|\geq \frac{\Delta(\delta+1)}{2}  \quad \mbox{if} \ \Delta (\delta+1) \ \mbox{is even},\quad
|E(G)|\geq \frac{\Delta(\delta+1)+1}{2}  \quad \mbox{if} \ \Delta  (\delta+1) \ \mbox{is odd},
$$
with equality if and only if $G \in \mathcal{G}_{\d,\D}$.
\end{proposition}


Let $\mathcal{I}$ be any topological index defined as
$$
\mathcal{I}(G)=\sum_{uv\in E(G)}h(d_u,d_v),
$$
where $h(x,y)$ is any positive symmetric function $h:\ZZ^+ \times \ZZ^+ \rightarrow \RR^+$.

Let $\mathcal{J}$ be any topological index defined as
$$
\mathcal{J}(G):=\sum_{u\in V(G)}\tilde{h}(d_u) ,
$$
where
$\tilde{h}$ is any positive function $\tilde{h}: \ZZ^+ \rightarrow \RR^+$.

\begin{definition} A graph $G$ with minimum degree $\delta$ and maximum degree $\Delta$ is \emph{minimal} for $\mathcal{J}$ if $\mathcal{J}(G) \le \mathcal{J}(\G)$
for every graph $\G$ with minimum degree $\delta$ and maximum degree $\Delta$.
\end{definition}

\begin{remark} \label{r:h}
Notice that every index $\mathcal{J}$ can be written also as an index $\mathcal{I}$. It suffices to define
$h(d_u,d_v):=\frac{\tilde{h}(d_u)}{d_u}+\frac{\tilde{h}(d_v)}{d_v}$ and then it can be readily seen that
$$\sum_{uv\in E(G)} \Big( \frac{\tilde{h}(d_u)}{d_u}+\frac{\tilde{h}(d_v)}{d_v}\Big) =\sum_{u\in V(G)}\tilde{h}(d_u).$$
\end{remark}

In \cite{MR2} appear several bounds for $\mathcal{I}$; in particular, \cite[Theorems 2.4 and 2.5]{MR2} give lower bounds for $\mathcal{I}$ when $h$ is a non-decreasing function (in the first variable).
Next, we are going to prove sharp inequalities for $\mathcal{J}$, when $\tilde{h}$ is a non-decreasing function, which can not be deduced from \cite[Theorems 2.4 and 2.5]{MR2}
and Remark \ref{r:h} (since there are non-decreasing functions $\tilde{h}$ such that their corresponding functions $h$ are strictly decreasing in the first variable).
Furthermore, the inequalities in Theorem \ref{t:tilde1} below are simpler than the ones in \cite[Theorems 2.4 and 2.5]{MR2}.

\begin{theorem} \label{t:tilde1}
Consider any integers $1 \le \d \le \D$ and assume that $\tilde{h}$ is non-decreasing.
If $G\in \mathcal{G}_{\d,\D}$, then it is minimal for $\mathcal{J}$, and thus
$$
\begin{aligned}
\mathcal{J}(G)\geq \D\tilde{h}(\d)+\tilde{h}(\D) & \qquad \text{if $\D(\d+1)$ is even},
\\
\mathcal{J}(G)\geq (\D-1)\tilde{h}(\d)+\tilde{h}(\d+1)+\tilde{h}(\D) & \qquad \text{if $\D(\d+1)$ is odd}.
\end{aligned}
$$
Moreover, if either
$\D(\d+1)$ is even and $\tilde{h}(\d)<\tilde{h}(\d+1)$ or if $\D(\d+1)$ is odd and $\tilde{h}(\d)<\tilde{h}(\d+1)<\tilde{h}(\d+2)$,
then $G$ is minimal for $\mathcal{J}$ if and only if $G\in \mathcal{G}_{\d,\D}$.
\end{theorem}

\begin{proof} Let $\G$ be a graph with minimum degree $\d$ and maximum degree $\D$. Since there is a vertex with degree $\D$, there are at least $\D$ additional vertices with degree at least $\d$.

If $\D(\d+1)$ is even, then $\mathcal{J}(\G)\geq \D\tilde{h}(\d)+\tilde{h}(\D)=\mathcal{J}(G)$ for every $G\in \mathcal{G}_{\d,\D}$.
If $\G\notin \mathcal{G}_{\d,\D}$, then there is some other vertex with degree at least $\d+1$ or some new vertex. In both cases, since $\tilde{h}(\d)<\tilde{h}(\d+1)$ and $\tilde{h}$ is non-decreasing and positive, it is trivial to check that
$\mathcal{J}(\G)>\mathcal{J}(G)$ for $G\in \mathcal{G}_{\d,\D}$.

If $\D(\d+1)$ is odd, then there are a vertex $v_0$ with degree $\D$ and at least $\D$ additional vertices.
Recall that there is a vertex different from $v_0$ with degree at least $\d+1$.
Thus, $\mathcal{J}(\G)\geq (\D-1)\tilde{h}(\d)+\tilde{h}(\d+1)+\tilde{h}(\D)=\mathcal{J}(G)$ for every $G\in \mathcal{G}_{\d,\D}$.

If $\G\notin \mathcal{G}_{\d,\D}$, then there is some other vertex with degree at least $\d+1$, a vertex with degree at least $\d+2$, or some new vertex.
In these cases, since $\tilde{h}(\d)<\tilde{h}(\d+1)<\tilde{h}(\d+2)$ and $\tilde{h}$ is non-decreasing and positive, it is trivial to check that
$\mathcal{J}(\G)>\mathcal{J}(G)$ for $G\in \mathcal{G}_{\d,\D}$.
\end{proof}

Our next goal is to prove Theorem \ref{t:tilde6}, a similar result to Theorem \ref{t:tilde1}, when $\tilde{h}$ is a non-increasing function.
The proof of Theorem \ref{t:tilde6} is harder that of Theorem \ref{t:tilde1},
and requires several previous results, namely,
Propositions \ref{p:H}, \ref{p:HHHH}, \ref{Prop: minimal D+2}, \ref{Prop: 211} and \ref{Prop: 213},
and Theorems \ref{t:tilde2},  \ref{t:tilde3}, \ref{t:tilde4} and \ref{t:tilde5}.


\begin{definition}
Let $\mathfrak{G}^{n}_{\d,\D}$ be the family of graphs with $n$ vertices, minimum degree $\d$ and maximum degree $\D$.
\end{definition}

\begin{proposition} \label{proposition211}
If $G\in \mathfrak{G}^{n}_{\D,\D}$, then it is minimal for $\mathcal{J}$ if and only if $n=\D+1$ and $G$ is the complete graph $K_{\D+1}$, and thus
$\mathcal{J}(G) \ge (\D+1)\tilde{h}(\D)$ for every graph $G$ with minimum and maximum degree $\D$.
\end{proposition}

\begin{proof}
It suffices to note that $\mathcal{J}(G)=n\tilde{h}(\D)$ with $n\geq \D+1$.
\end{proof}

By Proposition \ref{proposition211}, in order to find minimal graphs for $\mathcal{J}$, it suffices to consider the case $\d< \D$.

For any $1 \le \delta<\D$, let us define the family $\mathcal{H}_{\d,\D}$ of graphs with $\D+1$ vertices, where one of them has degree $\d$, $\d$ of them have degree $\D$ and
the other $\D-\d$ vertices have degree $\D-1$.

\begin{proposition} \label{p:H}
For every integers $1 \le \delta <\D$, the family $\mathcal{H}_{\d,\D}$ is non-empty.
\end{proposition}

\begin{proof}
It suffices to consider the graph $H_{\d,\D}$ obtained from
a complete graph with $\D$ vertices, $v_1,\dots,v_\D,$ by adding an extra vertex $w$ and edges $v_iw$ for every $1\leq i \leq \d$.
\end{proof}

Notice that for any $H\in \mathcal{H}_{\d,\D}$,
$$\mathcal{J}(H)=\tilde{h}(\d)+(\D-\d)\tilde{h}(\D-1)+ \d \tilde{h}(\D). $$

\begin{proposition} \label{p:HHHH}
Given any integers $1 \le \delta< \D$, there is a unique graph (up to isomorphism), $H_{\d,\D}$, in $\mathcal{H}_{\d,\D}$,
where $H_{\d,\D}$ is the graph in the proof of Proposition \ref{p:H}.
\end{proposition}

\begin{proof} Let $G$ be a graph in $\mathcal{H}_{\d,\D}$.
Denote by $w$ a vertex with degree $\d$ (which is unique if $\d<\D-1$) and by $v_1,\dots,v_\D$ the other vertices.
Without loss of generality we can assume that $v_1,\dots,v_\d$ are the neighbors of $w$.
Since $v_{\d+1},\dots,v_\D$ are not neighbors of $w$, they have degree at most $\D-1$; hence, they have degree $\D-1$.
Thus, $v_1,\dots,v_\d$ have degree $\D$.
Therefore, the subgraph of $G$ induced by $v_1,\dots,v_\D$ is a complete graph, and $G$ is isomorphic to $H_{\d,\D}$.
\end{proof}

\begin{theorem} \label{t:tilde2}
Consider any integers $1 \le \delta< \D$.
If $\tilde{h}$ is non-increasing, then $H_{\d,\D}$ is minimal for $\mathcal{J}$ on $\mathfrak{G}^{\D+1}_{\d,\D}$. Moreover,
if $\tilde{h}(\D-2)>\tilde{h}(\D-1)>\tilde{h}(\D)$, then $G$ is minimal for $\mathcal{J}$ on $\mathfrak{G}^{\D+1}_{\d,\D}$ if and only if $G=H_{\d,\D}$.
\end{theorem}

\begin{proof} Suppose $G\in \mathfrak{G}^{\D+1}_{\d,\D}$. Since there is a vertex with degree $\d$ and $\D+1$ vertices, there are $\D-\d$ vertices with degree at most $\D-1$. Thus,
$$\mathcal{J}(G)\geq \d \tilde{h}(\D)+(\D-\d)\tilde{h}(\D-1)+\tilde{h}(\d)=\mathcal{J}(H_{\d,\D}).$$

Assume that $G\in \mathfrak{G}^{\D+1}_{\d,\D}$. If $G\neq H_{\d,\D}$, then either there is one less vertex with degree $\D$ or one more vertex with degree less than $\D-1$. Thus, since $\tilde{h}(\D-2)>\tilde{h}(\D-1)>\tilde{h}(\D)$ and $\tilde{h}$ is non-increasing, we have
$\mathcal{J}(G)>\mathcal{J}(H_{\d,\D})$.
\end{proof}

For every $1 \le \delta <\D$, let us define the family $\mathcal{K}_{\d,\D}$ of graphs with $\D+2$ vertices, where one of them has degree $\d$ and we have either:
\begin{itemize}
	\item if $\d$ is even, $\D+1$ vertices have degree $\D$,
	\item if $\delta$ is odd, $\D$ vertices have degree $\D$ and the last vertex has degree $\D-1$.
\end{itemize}

\begin{proposition} \label{Prop: minimal D+2} For every integers $1 \le \delta <\D$, the family $\mathcal{K}_{\d,\D}$ is non-empty.
\end{proposition}

\begin{proof} Consider a complete graph with $\D+1$ vertices, $v_1,\dots,v_{\D+1},$ and:
\begin{itemize}
	\item If $\delta$ is even, consider $\delta$ vertices $v_1,\dots,v_\d$ and remove the edges $v_{2i-1}v_{2i}$ for every $1\leq i \leq \frac{\d}{2}$. Finally, add an extra vertex $w$ and edges $v_jw$ for every $1\leq j \leq \d$. Let us denote this graph by $K_{\d,\D}$.
	\item If $\delta$ is odd, consider $\delta+1$ vertices $v_1,\dots,v_{\d+1}$ and eliminate the edges $v_{2i-1}v_{2i}$ for every $1\leq i \leq \frac{\d+1}{2}$. Finally, add an extra vertex $w$ and edges $v_jw$ for every $1\leq j \leq \d$. Let us denote this graph by $K^1_{\d,\D}$.
\end{itemize}
\end{proof}

Notice that for any $K\in \mathcal{K}_{\d,\D}$,
\begin{itemize}
	\item if $\delta$ is even, $\mathcal{J}(K)= (\D+1) \tilde{h}(\D)+\tilde{h}(\d). $
	\item if $\delta$ is odd, $\mathcal{J}(K)= \D \tilde{h}(\D)+\tilde{h}(\D-1)+\tilde{h}(\d). $
\end{itemize}

\begin{proposition} \label{Prop: 211}
For any integers $1 \le \delta\leq \D$, the following statements hold:
\begin{itemize}
	\item if $\delta$ is even, then there is a unique graph (up to isomorphism), $K_{\d,\D}$, in $\mathcal{K}_{\d,\D}$,
	\item if $\delta$ is odd, then there exist only two graphs (up to isomorphism), $K^1_{\d,\D}$ and $K^2_{\d,\D}$, in $\mathcal{K}_{\d,\D}$,
\end{itemize}
where $K_{\d,\D}$ and $K^1_{\d,\D}$ are the graphs in the proof of Proposition \ref{Prop: minimal D+2}
\end{proposition}

\begin{proof} Let $G$ be a graph in $\mathcal{K}_{\d,\D}$. Denote by $w$ the vertex with degree $\d$  and by $v_1,\dots,v_{\D+1}$ the other vertices. Without loss of generality we can assume that $v_1,\dots,v_\d$ are the neighbors of $w$.

Assume first that $\d$ is even.
Thus, $v_{\d+1},\dots,v_{\D+1}$ are adjacent to every vertex distinct from $w$.
Now, in $\{v_1,\dots,v_\d\}$ every vertex is adjacent to every vertex in this set except one.
We may assume without loss of generality that $v_{2i-1}v_{2i}\notin E(G)$ for every $1\leq i \leq \frac{\d}{2}$.
Hence, $G$ is isomorphic to the graph $K_{\d,\D}$. 

Assume now that $\d$ is odd and the vertex with degree $\D-1$ is not adjacent to $w$.
Thus, we can assume that $v_{\d+1},\dots,v_{\D}$ have degree $\D$ (and so, they are adjacent to every vertex distinct from $w$),
and $v_{\D+1}$ has degree $\D-1$ and it is adjacent to every vertex distinct from $w$ and $v_\d$.
Then, $v_\d$ is adjacent to every vertex distinct from $v_{\D+1}$.
Hence, we may assume without loss of generality that $v_{2i-1}v_{2i}\notin E(G)$ for every $1\leq i \leq \frac{\d-1}{2}$.
Thus, $G$ is isomorphic to the graph $K^1_{\d,\D}$. 

Assume now that $\d$ is odd and the vertex with degree $\D-1$ is adjacent to $w$.
Thus, $v_{\d+1},\dots,v_{\D+1}$ are adjacent to every vertex distinct from $w$, and we can assume that $v_\d$ has degree $\D-1$ and $v_{\d-2}v_\d, v_{\d-1}v_\d \notin E(G)$.
Hence, we may assume without loss of generality that $v_{2i-1}v_{2i}\notin E(G)$ for every $1\leq i \leq \frac{\d-3}{2}$, and therefore $G$ is isomorphic to a unique graph $K^2_{\d,\D}$.
\end{proof}

\begin{theorem} \label{t:tilde3}
Consider any integers $1 \le \delta< \D$.
If $\tilde{h}$ is non-increasing and $G\in \mathcal{K}_{\d,\D}$, then $G$ is minimal for $\mathcal{J}$ on $\mathfrak{G}^{\D+2}_{\d,\D}$.
Moreover, if either $\d$ is even and $\tilde{h}(\D-1)>\tilde{h}(\D)$ or $\d$ is odd and
$\tilde{h}(\D-2)>\tilde{h}(\D-1)>\tilde{h}(\D)$, then $G$ is minimal for $\mathcal{J}$ on $\mathfrak{G}^{\D+2}_{\d,\D}$ if and only if $G\in \mathcal{K}_{\d,\D}$.
\end{theorem}

\begin{proof} Suppose $G\in \mathfrak{G}^{\D+2}_{\d,\D}$.
If $\d$ is even, then
$$\mathcal{J}(G)\geq (\D+1) \tilde{h}(\D)+\tilde{h}(\d)=\mathcal{J}(K_{\d,\D}).$$
If $\d$ is odd, then at least one of the vertices has degree at most $\D-1$, and therefore
$$
\mathcal{J}(G)\geq \D \tilde{h}(\D)+\tilde{h}(\D-1)+\tilde{h}(\d)=\mathcal{J}(K^{i}_{\d,\D}) \quad \mbox{ for } i=1,2.
$$

Assume that $G\in \mathfrak{G}^{\D+2}_{\d,\D}$ and $G\notin \mathcal{K}_{\d,\D}$. If $\d$ is even, then there are at least two vertices with degree less than $\D$ (one of them with degree $\d$). Thus, since  $\tilde{h}(\D-1)>\tilde{h}(\D)$ and $\tilde{h}$ is non-increasing,
$$\mathcal{J}(G)\geq \D \tilde{h}(\D)+\tilde{h}(\D-1)+\tilde{h}(\d)>\mathcal{J}(K_{\d,\D}).$$
If $\d$ is odd, then either there are two vertices with degree less than $\D-1$ (assuming $\d\leq \D-2$) or there are three vertices with degree less than $\D$. Thus, since  $\tilde{h}(\D-2)>\tilde{h}(\D-1)>\tilde{h}(\D)$ and $\tilde{h}$ is non-increasing, either
$$\mathcal{J}(G)\geq \D \tilde{h}(\D)+\tilde{h}(\D-2)+\tilde{h}(\d)>\mathcal{J}(K^{i}_{\d,\D}),$$
or
$$\mathcal{J}(G)\geq (\D-1) \tilde{h}(\D)+2\tilde{h}(\D-1)+\tilde{h}(\d)>\mathcal{J}(K^{i}_{\d,\D}).$$
\end{proof}

If $1\le \d<\D$ and $\d\D$ is odd, then we define $\mathcal{L}_{\d,\D}$ as the family of graphs with a unique vertex with degree $\d$ and $\D+2$ vertices with degree $\D$.

\begin{proposition}  \label{Prop: 213}
If $1\le \d<\D$ and $\d\D$ is odd, then the family $\mathcal{L}_{\d,\D}$ is non-empty.
\end{proposition}

\begin{proof} Consider a graph $G_0$ with $\D+2$ vertices, $v_1,\dots,v_{\D+2}$, which is the complement of a cycle. Add to $G_0$ an extra vertex $w$ and edges $v_iw$ for every $1\leq j \leq \d$ and an edge $v_{2k}v_{2k+1}$ for every
$\frac{\d+1}{2} \leq k \leq \frac{\D+1}{2}$.
The obtained graph belongs to $\mathcal{L}_{\d,\D}$.
\end{proof}

Notice that for any $L\in \mathcal{L}_{\d,\D}$, we have
$$\mathcal{J}(L)=(\D+2)\tilde{h}(\D)+ \tilde{h}(\d). $$


\begin{theorem} \label{t:tilde4}
If $1\le \d<\D$, $\d\D$ is odd, $\tilde{h}$ is non-increasing and $G\in \mathcal{L}_{\d,\D}$, then $G$ is minimal
for $\mathcal{J}$ on $\mathfrak{G}^{\D+3}_{\d,\D}$. Moreover,
if $\tilde{h}(\D-1)>\tilde{h}(\D)$, then $G$ is minimal
for $\mathcal{J}$ on $\mathfrak{G}^{\D+3}_{\d,\D}$ if and only if $G\in \mathcal{L}_{\d,\D}$.
\end{theorem}

\begin{proof} Suppose $G\in \mathfrak{G}^{\D+3}_{\d,\D}$. Since there is a vertex with degree $\d$, we have
$$\mathcal{J}(G)\geq (\D+2) \tilde{h}(\D)+\tilde{h}(\d)=\mathcal{J}(L) \mbox{ for any } L \in \mathcal{L}_{\d,\D}.$$

Assume that $G\in \mathfrak{G}^{\D+3}_{\d,\D}$. If $G\notin \mathcal{L}_{\d,\D}$, then there is one less vertex with degree $\D$. Thus, since $\tilde{h}(\D-1)>\tilde{h}(\D)$ and $\tilde{h}$ is non-increasing,
$\mathcal{J}(G)>\mathcal{J}(L)$ for any  $L \in \mathcal{L}_{\d,\D}$.
\end{proof}

\begin{theorem} \label{t:tilde5}
Consider any integers $1 \le \delta< \D$.
Assume that $\tilde{h}$ is non-increasing and $G$ is minimal for $\mathcal{J}$. The following facts hold:
\begin{itemize}
	\item if $\d \D$ is even, then $\D+1\leq |V(G)|\leq \D+2$,
	\item if $\d \D$ is odd and $2\tilde{h}(\D)>\tilde{h}(\D-1)$, then $\D+1\leq |V(G)|\leq \D+2$,
	\item if $\d \D$ is odd and $2\tilde{h}(\D)\leq \tilde{h}(\D-1)$, then $\D+1\leq |V(G)|\leq \D+3$.
\end{itemize}
\end{theorem}

\begin{proof} The lower bound is immediate since there is a vertex with degree $\D$. Suppose that $G$ is a graph with at least $\D+3$ vertices. Then, since one of them has degree $\d$ and
$\tilde{h}$ is non-increasing,
$$
\mathcal{J}(G)
\geq (\D+2) \tilde{h}(\D)+\tilde{h}(\d)
> (\D+1) \tilde{h}(\D)+\tilde{h}(\d).
$$

Consequently, if $\d$ is even, then $\mathcal{J}(G)> \mathcal{J}(K_{\d,\D})$.

If $\d$ is odd and $\D$ is even, there is other vertex with odd degree and, in particular, with degree less than $\D$. Therefore,
$$
\mathcal{J}(G)
\geq (\D+1) \tilde{h}(\D)+\tilde{h}(\D-1)+\tilde{h}(\d)
> \mathcal{J}(K)
$$
for any $K\in \mathcal{K}_{\d,\D}.$

If $\d \D$ is odd and $2\tilde{h}(\D)>\tilde{h}(\D-1)$, then
$$
\mathcal{J}(G)
\geq (\D+2) \tilde{h}(\D)+\tilde{h}(\d)
> \D \tilde{h}(\D)+\tilde{h}(\D-1)+\tilde{h}(\d)
= \mathcal{J}(K)
$$
for any $K\in \mathcal{K}_{\d,\D}.$

Finally, if $\d \D$ is odd and $2\tilde{h}(\D)\leq \tilde{h}(\D-1)$, then suppose that $G$ is a graph with at least $\D+4$ vertices.
Thus,
$$
\mathcal{J}(G)\geq (\D+3) \tilde{h}(\D)+\tilde{h}(\d)
> (\D+2) \tilde{h}(\D)+\tilde{h}(\d)
= \mathcal{J}(L)
$$
for any $L\in \mathcal{L}_{\d,\D}.$
\end{proof}

The previous results allow to obtain the following sharp lower bounds for $\mathcal{J}$
when $\tilde{h}$ is a non-increasing function.

\begin{theorem} \label{t:tilde6}
Consider any integers $1 \le \delta< \D$.
Assume that $\tilde{h}$ is non-increasing and $G$ is a graph with minimum degree $\d$ and maximum degree $\D$.
The following inequalities hold:
\begin{itemize}
	\item if $\d$ is even, then
$$
\mathcal{J}(G)
\ge \min \big\{\d \tilde{h}(\D)+(\D-\d)\tilde{h}(\D-1)+\tilde{h}(\d),\,
(\D+1) \tilde{h}(\D)+\tilde{h}(\d) \big\},
$$
	\item if either $\d$ is odd and $\D$ is even or $\d \D$ is odd and $2\tilde{h}(\D)>\tilde{h}(\D-1)$, then
$$
\mathcal{J}(G)
\ge \min \big\{\d \tilde{h}(\D)+(\D-\d)\tilde{h}(\D-1)+\tilde{h}(\d),\,
\D \tilde{h}(\D)+\tilde{h}(\D-1)+\tilde{h}(\d) \big\},
$$
	\item if $\d \D$ is odd and $2\tilde{h}(\D)\leq \tilde{h}(\D-1)$, then
$$
\mathcal{J}(G)
\ge \min \big\{\d \tilde{h}(\D)+(\D-\d)\tilde{h}(\D-1)+\tilde{h}(\d),\,
(\D+2) \tilde{h}(\D)+\tilde{h}(\d) \big\}.
$$
\end{itemize}
\end{theorem}

\begin{proof}
Let $G$ be a graph with minimum degree $\d$ and maximum degree $\D$.

If $\d$ is even and $\G$ is minimal for $\mathcal{J}$, then Theorem \ref{t:tilde5} gives $\D+1\leq |V(\G)|\leq \D+2$.
Hence, Theorems \ref{t:tilde2} and \ref{t:tilde3} give
$$
\begin{aligned}
\mathcal{J}(G)
& \ge \min \big\{ \mathcal{J}(H_{\d,\D}),\, \mathcal{J}(K_{\d,\D})\big\}
\\
& = \min \big\{\d \tilde{h}(\D)+(\D-\d)\tilde{h}(\D-1)+\tilde{h}(\d),\,
(\D+1) \tilde{h}(\D)+\tilde{h}(\d) \big\}.
\end{aligned}
$$

If either $\d$ is odd and $\D$ is even or $\d \D$ is odd and $2\tilde{h}(\D)>\tilde{h}(\D-1)$, and $\G$ is minimal for $\mathcal{J}$, then Theorem \ref{t:tilde5} gives $\D+1\leq |V(\G)|\leq \D+2$.
Hence, Theorems \ref{t:tilde2} and \ref{t:tilde3} give
$$
\begin{aligned}
\mathcal{J}(G)
& \ge \min \big\{ \mathcal{J}(H_{\d,\D}),\, \mathcal{J}(K_{\d,\D}^i)\big\}
\\
& = \min \big\{\d \tilde{h}(\D)+(\D-\d)\tilde{h}(\D-1)+\tilde{h}(\d),\,
\D \tilde{h}(\D)+\tilde{h}(\D-1)+\tilde{h}(\d) \big\}.
\end{aligned}
$$

If $\d \D$ is odd and $2\tilde{h}(\D) \le \tilde{h}(\D-1)$, and $\G$ is minimal for $\mathcal{J}$, then Theorem \ref{t:tilde5} gives $\D+1\leq |V(\G)|\leq \D+3$.
Hence, Theorems \ref{t:tilde2}, \ref{t:tilde3} and \ref{t:tilde4} give, if $L \in \mathcal{L}_{\d,\D}$,
$$
\begin{aligned}
\mathcal{J}(G)
& \ge \min \big\{ \mathcal{J}(H_{\d,\D}),\, \mathcal{J}(K_{\d,\D}),\, \mathcal{J}(L) \big\}
\\
& = \min \big\{\d \tilde{h}(\D)+(\D-\d)\tilde{h}(\D-1)+\tilde{h}(\d),\,
\D \tilde{h}(\D)+\tilde{h}(\D-1)+\tilde{h}(\d),\,
(\D+2) \tilde{h}(\D)+\tilde{h}(\d) \big\}.
\end{aligned}
$$
Since $2\tilde{h}(\D) \le \tilde{h}(\D-1)$, we have $(\D+2) \tilde{h}(\D)+\tilde{h}(\d) \le \D \tilde{h}(\D)+\tilde{h}(\D-1)+\tilde{h}(\d)$
and thus
$$
\mathcal{J}(G)
\ge \min \big\{\d \tilde{h}(\D)+(\D-\d)\tilde{h}(\D-1)+\tilde{h}(\d),\,
(\D+2) \tilde{h}(\D)+\tilde{h}(\d) \big\}.
$$
\end{proof}

\begin{remark}\label{r:log}
Let us consider the multiplicative index
$$
\mathcal{J}'(G):= \prod_{u \in V(G)} \tilde{h}'(d_u),
$$
where $\tilde{h}'$ is any function $\tilde{h}':\ZZ^+ \rightarrow (1,\infty)$.
Since $\log \mathcal{J}'(G)= \sum_{u \in V(G)} \log \tilde{h}'(d_u)$,
Theorems \ref{t:tilde1} and \ref{t:tilde6} can be applied to $\mathcal{J}'$.
\end{remark}

\section{Inequalities for some particular indices}

In this last section, we want to apply the previous results to some particular indices.

Theorems \ref{t:tilde1} and \ref{t:tilde6} have the following consequences for the first variable Zagreb index.

\begin{theorem} \label{t:tilde1bis}
Let us consider $\a \in \RR$ with $\a>0$, any integers $1 \le \delta \le \Delta$, and $G$ a graph with minimum degree $\d$ and maximum degree $\D$.
The following sharp inequalities hold:

If $\D(\d+1)$ is even, then
$$
M_1^{\a}(G)\geq \D \d^{2\a} + \D^{2\a}.
$$

If $\D(\d+1)$ is odd, then
$$
M_1^{\a}(G)\geq (\D-1) \d^{2\a}+(\d+1)^{2\a} + \D^{2\a}.
$$
\end{theorem}

\begin{theorem} \label{t:tilde6bis}
Let us consider $\a \in \RR$ with $\a<0$, any integers $1 \le \delta < \Delta$, and $G$ a graph with minimum degree $\d$ and maximum degree $\D$.
The following sharp inequalities hold:
\begin{itemize}
	\item if $\d$ is even, then
$$
M_1^{\a}(G)
\ge \min \big\{\d \D^{2\a} + (\D-\d) (\D-1)^{2\a}+ \d^{2\a},\,
(\D+1) \D^{2\a}+ \d^{2\a} \big\},
$$
	\item if either $\d$ is odd and $\D$ is even or $\d \D$ is odd and $2 \D^{2\a} > (\D-1)^{2\a}$, then
$$
M_1^{\a}(G)
\ge \min \big\{\d \D^{2\a} + (\D-\d) (\D-1)^{2\a}+ \d^{2\a},\,
\D^{2\a+1} + (\D-1)^{2\a}+ \d^{2\a} \big\},
$$
	\item if $\d \D$ is odd and $2 \D^{2\a}\leq (\D-1)^{2\a}$, then
$$
M_1^{\a}(G)
\ge \min \big\{\d \D^{2\a} + (\D-\d) (\D-1)^{2\a}+ \d^{2\a},\,
(\D+2) \D^{2\a} + \d^{2\a} \big\}.
$$
\end{itemize}
\end{theorem}

Remark \ref{r:log} allows to apply Theorem \ref{t:tilde1} to the following multiplicative indices (if $\d\ge 2$):
the \emph{first multiplicative Zagreb index} \cite{Gutman8}, defined as
$$
\Pi_1(G) = \prod_{u\in V (G)} d_u^2,
$$
the \emph{Narumi-Katayama index} defined in \cite{NK} as
$$
NK(G) = \prod_{u\in V (G)} d_u,
$$
(note that $\Pi_1(G) = NK(G)^2$ for every graph $G$) and the \emph{modified Narumi-Katayama index} introduced in \cite{GSG} as
$$
NK^*(G) = \prod_{u\in V (G)} d_u^{d_u}.
$$

\begin{theorem} \label{t:product}
Let us consider any integers $1 \le \delta < \Delta$, and $G$ a graph with minimum degree $\d$ and maximum degree $\D$.
The following sharp inequalities hold:

If $\D(\d+1)$ is even, then
$$
\Pi_1(G) \geq \d^{2\D}\D^{2},
\quad
NK(G) \geq \d^{\D}\D ,
\quad
NK^*(G) \geq \d^{\d\D}\D^{\D}.
$$

If $\D(\d+1)$ is odd, then
$$
\Pi_1(G) \geq \d^{2(\D-1)}(\d+1)^2\D^{2},
\quad
NK(G) \geq \d^{\D-1}(\d+1)\D ,
\quad
NK^*(G) \geq \d^{\d(\D-1)}(\d+1)^{\d+1}\D^{\D}.
$$
\end{theorem}

\smallskip

\emph{Acknowledgement:}
%
This work was partially supported by the grants from Ministerio de Econom{\'\i}a y Competititvidad, Agencia Estatal de
Investigación (AEI) and Fondo Europeo de Desarrollo Regional (FEDER) (MTM 2015-63612P, MTM 2016-78227-C2-1-P and MTM 2015-69323-REDT), Spain.

\end{document}